\documentclass[12pt]{article}

\usepackage{amsthm}
\usepackage{amssymb}
\usepackage{amsmath}
\usepackage{amsfonts}
\usepackage{times}
\usepackage{color}
\usepackage{verbatim}

\usepackage[paper=a4paper, top=2.5cm, bottom=2.5cm, left=2.5cm, right=2.5cm]{geometry}

\def\qed{\hfill $\vcenter{\hrule height .3mm
\hbox {\vrule width .3mm height 2.1mm \kern 2mm \vrule width .3mm
height 2.1mm} \hrule height .3mm}$ \bigskip}

\def \RR {\mathbb R}

\def \EE {\mathbb E}

\def \PP {\mathbb P}
\def \QQ {\mathbb Q}

\def \cE {\mathcal E}

\def \cL {\mathcal L}

\def \FF {\mathcal{F}}

\def \Var {\mathrm{Var}}

\newtheorem{theorem}{Theorem}
\newtheorem{lemma}[theorem]{Lemma}

\newtheorem{corollary}[theorem]{Corollary}
\theoremstyle{definition}
\newtheorem{definition}[theorem]{Definition}
\theoremstyle{remark}
\newtheorem{remark}[theorem]{Remark}

\newcommand{\corFK}{}

\def\corON{}
\def\RE{}
\long\def\symbolfootnotetext[#1]#2{\begingroup
\def\thefootnote{\fnsymbol{footnote}}\footnotetext[#1]{#2}\endgroup}

\begin{document}

\title{A Spectral Condition for Spectral Gap: Fast Mixing in High-Temperature Ising Models}
\author{Ronen Eldan\thanks{Weizmann Institute of Science. Supported by a European Research Council Starting Grant (ERC StG) grant agreement no. 803084 and by an Israel Science Foundation grant no. 715/16. Email: ronen.eldan@weizmann.ac.il.} \and Frederic Koehler\thanks{Massachusetts Institute of Technology. This work was supported in part by NSF CAREER Award CCF-1453261, NSF Large CCF-1565235, Ankur Moitra’s ONR Young Investigator Award, and European Research Council (grant no. 803084). Email: fkoehler@mit.edu (corresponding author).} \and Ofer Zeitouni\thanks{Weizmann Institute of Science. This project has received funding from the European Research Council (ERC) under the European Union's Horizon 2020 research and innovation programme (grant agreement No. 692452). Email: ofer.zeitouni@weizmann.ac.il.}}
\date{}
\maketitle
\begin{abstract}
  \corON{We prove that Ising models on the hypercube with general quadratic 
  interactions satisfy a Poincar\'{e} inequality with respect
  to the natural Dirichlet form corresponding to Glauber dynamics, as soon as 
  the operator norm of the interaction matrix is smaller than $1$. 
  The inequality
implies a control on the mixing time of the Glauber dynamics. Our techniques rely on a localization procedure which establishes a structural result, stating that Ising measures may be decomposed into a mixture of measures with quadratic potentials of rank one, and provides a framework for proving concentration bounds for high temperature Ising models.}
\end{abstract}

\section{Introduction}
In this paper we study the high temperature behavior of the Sherrington-Kirkpatrick model and more general Ising models, especially with regards to mixing of the Glauber dynamics (i.e. Gibbs sampling) chain. More precisely, if $\mu$ is the uniform measure over the hypercube $\{\pm 1\}^n$, we consider a general Ising model of the form
\begin{equation}\label{eqn:ising}
  \frac{\corON{d\nu_0}}{d\mu}(x) =\frac1Z \exp\left(\frac{1}{2}\langle x, J x \rangle + \langle h, x \rangle\right)
\end{equation}
for an arbitrary symmetric quadratic interaction matrix $J$ and \emph{external field} $h \in \mathbb{R}^n$, \corON{where $Z$, the partition function,
is a normalization constant.} Because \corON{the evaluation of the partition function $Z$ is a difficult computational task,} 
%is specified up to 
%\corON{the hard to compute }
%an unknown 
%constant of proportionality, 
in practice samples from \eqref{eqn:ising} are generally constructed by simulating a Markov chain such as the \emph{Glauber dynamics}, where at each step (in discrete time) a \corON{site} $i$ is chosen uniformly at random from $[n]$ and the \corON{random}
spin $X_i$ is resampled from its conditional law given $X_{\sim i}$. \corON{(Here and throughout, we write $x_{\sim i}$ for the
collections $\{x_j\}_{j\neq i}$, with similar notation for $X_{\sim i}$.)}

The behavior of Glauber dynamics in the Ising model is a classical and well-studied topic with rich connections to structural properties of the Gibbs measure $\nu$ and concentration of measure. As far as sufficient conditions for fast mixing are concerned, one of the most general and well-known situations where rapid mixing is guaranteed is under \emph{Dobrushin's uniqueness condition} \cite{dobrushin-uniqueness}, which requires that $\|J\|_{\infty \to \infty} < 1$ or equivalently that $\sum_{j} |J_{ij}| < 1$ for all rows $i$. Unfortunately, even though there exist situations where this bound is tight (the mean-field/Curie-Weiss model), in other situations of interest this bound is far from tight. 

One notable model where Dobrushin's condition is 
\corON{not satisfied at interesting situations}
%believed to be weak 
is the celebrated \emph{Sherrington-Kirkpatrick} (SK) model from spin glass theory \cite{talagrand2010mean}. In the SK model, $J$ is given by a rescaled matrix from the Gaussian Orthogonal Ensemble so that $J$ is symmetric with off-diagonal entries $J_{ij} \sim \mathcal{N}(0, \beta^2/n)$, where $\beta > 0$ is a parameter specifying the \emph{inverse temperature} of the model. Here the expected $\ell_1$ norm of a row of $J$ is on the order of $\beta \sqrt{n}$, so that Dobrushin's uniqueness condition only holds under the restrictive condition $\beta = O(1/\sqrt{n})$. Nevertheless, it is expected that in reality the Glauber dynamics are actually fast mixing for all sufficiently small constant $\beta = O(1)$ (i.e. not shrinking with $n$). 
\corON{We indeed prove this below, see Theorem \ref{theo-11} and Section \ref{sec-examples}.} 

In the classical case of ferromagnetic Ising models on a lattice, \corON{it is}
known that there are close connections between rapid mixing of the Glauber dynamics and  functional inequalities such as the log-Sobolev inequality. In a recent breakthrough result, Bauerschmidt and Bodineau \cite{bauerschmidt2019very} proved a form of the log-Sobolev inequality for the SK model at sufficiently high temperature ($\beta < 0.25$). More precisely, they proved that if $J$ is a positive semidefinite matrix \corON{of operator norm $\|J\|_{OP}$,} 
then \corON{for any probability measure $\rho$ on $\{\pm 1\}^n$,}
\begin{equation}\label{eqn:log-sobolev-bb}
  \corON{D(\rho || \nu_0)} \lesssim \frac{1}{1 - \|J\|_{OP}} \sum_{i = 1}^n \EE_{\nu} \left|\partial_i \sqrt{\frac{d\rho}{\corON{d\nu_0}}}(X)\right|^2 
\end{equation}
for any model of the form \eqref{eqn:ising}, where $D(\rho ||\corON{ \nu_0}) = \EE_{\rho} \log \frac{d\rho}{\corON{d\nu_0}}$ is the \emph{relative entropy} and $\partial_i f(x) = f(x_{\sim i}, x_i = 1) - f(x_{\sim i}, x_i = -1)$ is the discrete gradient on the hypercube. By a standard argument (see e.g. \cite{ledoux1999concentration,van2014probability}), this implies the following Poincar\'e-type inequality
\begin{equation}\label{eqn:poincare-bb}
  \Var(\varphi) \lesssim \frac{1}{1 - \|J\|_{OP}} \sum_{i = 1}^n \EE_{\corON{\nu_0}} \left|\partial_i \varphi(X)\right|^2 
\end{equation}
as well. Their proof is based upon \corON{an}
explicit decomposition of the measure \corON{$\nu_0$} into a mixture of product measures. 

However, in the case of the SK model the estimates \eqref{eqn:log-sobolev-bb} and \eqref{eqn:poincare-bb} are not known to imply polynomial time bounds on the mixing time (or relaxation time) of Glauber dynamics. The reason is a subtle discrepancy between different notions of discrete gradients. \RE{A simple example which illustrates this is the uniform measure $\mu_{\mathrm{even}}$ on the set of vertices with even parity, $\{(x_1,\dots,x_n \in \{\pm 1\}^n;  \prod_i x_i = 1 \}$. It is not hard to check that the right hand side of \eqref{eqn:poincare-bb} remains unchanged if the measure $\mu_{\mathrm{even}}$ is replaced by the uniform measure $\mu$, therefore, if we apply the law of total variance, we see that $\mu_{\mathrm{even}}$ satisfies a Poincar\'e-type inequality of this form. On the other hand, the Glauber dynamics with respect to this measure is trapped at one vertex.} 

For the Glauber dynamics, having spectral gap $\gamma$ (or equivalently, relaxation time $1/\gamma$) is equivalent to the following Poincar\'e inequality (see e.g. \cite{van2014probability}):
\begin{equation}\label{eqn:poincare-glauber}
  \Var(\varphi) \le \frac{1}{\gamma} \cE_{\corON{\nu_0}}(\varphi,\varphi) := \frac{1}{\gamma} \EE_{\corON{\nu_0}} \sum_{i = 1}^n (\EE_{\corON{\nu_0}}[\varphi(X) | X_{\sim i}] - \varphi(X))^2
\end{equation}
where the rhs $\cE_{\corON{\nu_0}}(\varphi,\varphi)$ 
here is the \emph{Dirichlet form} corresponding to the continuous time Glauber dynamics. \corFK{For distributions with full support on the hypercube $\{\pm 1\}^n$, the right hand side of \eqref{eqn:poincare-bb} can be realized as the Dirichlet form for a certain dynamics in continuous time, see \cite[Equation (3.6)]{martinelli1999lectures}, but these dynamics can have a rate that is super-polynomial (see discussion below).}
  Similarly, the canonical log-Sobolev inequality for the Glauber semigroup in the sense of Gross 
  \cite{gross1975logarithmic}, which implies \eqref{eqn:poincare-glauber} as well as rapid mixing, 
replaces the sum on right hand side of \eqref{eqn:log-sobolev-bb} 
by the Dirichlet form $\cE_{\corON{\nu_0}}\left(\sqrt{\frac{d\rho}{d\corON{\nu_0}}}, \sqrt{\frac{d\rho}{d\corON{\nu_0}}}\right)$. For some models, %(e.g. those under Dobrushin's uniqueness regime), % technically inaccurate with large external fields
the discrepancy between the rhs of \eqref{eqn:poincare-bb} and \eqref{eqn:poincare-glauber} is at most a constant factor and so the difference between the Dirichlet form \corFK{of the Glauber dynamics} and the $\ell_2$-norm of the usual discrete gradient can be disregarded. 

Unfortunately, for the SK model it turns out the right hand side of \eqref{eqn:poincare-bb} can be size $e^{\Theta(\beta \sqrt{n})}$ larger than the right hand side of \eqref{eqn:poincare-glauber} in some simple examples. (We give such an example in Appendix~\ref{apdx:dirichlet-comparison}.) On the other hand, it is not hard to show that in the reverse direction, the rhs of \eqref{eqn:poincare-glauber} is never bigger than \eqref{eqn:poincare-bb} by more than a constant factor, so that \eqref{eqn:poincare-glauber} is a stronger estimate. \corFK{This inequivalence of the two Dirichlet forms reflects the fact that the rate of the continuous-time dynamics corresponding to \eqref{eqn:poincare-bb} can be as large as $e^{\Theta(\beta \sqrt{n})}$.}
%It is worthwhile to note that the right hand side of equation \eqref{eqn:poincare-bb} does correspond to a certain dynamics in continuous time, see \cite[Equation (3.6)]{martinelli1999lectures}.

The main result of this paper is \corON{a proof of}
%to prove 
the Poincar\'e inequality \eqref{eqn:poincare-glauber} with $\gamma = 1 - \|J\|_{OP}$ (for $J$ psd, as before), from which we obtain polynomial bounds on the mixing time of the Glauber dynamics. It is unclear how to obtain such a result from the product measure decomposition used to prove \eqref{eqn:log-sobolev-bb}, so a key technical idea in our work is the construction of a new 
\corON{decomposition of the measure $\nu$}
as a mixture of \emph{rank-one} Ising models (i.e. where $J$ has rank one). This can be thought of as a natural analogue of the \emph{needle decomposition} used in convex geometry \cite{kannan1995isoperimetric}. \RE{A structural theorem of this form, however not used directly in our result, is formulated in Section \ref{sec:needle} below}. The needle decomposition itself is generated by a natural stochastic process (a version of stochastic localization \cite{eldan2013thin}) and the smooth nature of the decomposition allows us to explicitly analyze the evolution of the Dirichlet form along this process, allowing us to prove the result. 

\corON{In the next section, 
  we formulate and prove our basic Poincar\'{e} inequality,
Theorem \ref{thm:poincare}.} \corFK{The related Appendix~\ref{apdx:dirichlet-comparison} discusses  the inequivalence between the Dirichlet forms $\cE_\nu(\varphi,\varphi)$ and 
$\EE|\nabla \varphi(X)|^2$.}
\corON{Section \ref{sec-4} is devoted to the estimate
on mixing time, Theorem \ref{theo-11}. In Section \ref{sec:needle} we outline a structural theorem in the spirit of the needle-decompositions mentioned above. Finally, Section \ref{sec-examples} is devoted to examples. }

\textbf{Acknowledgements.} We would like to thank Fanny Augeri for enlightening discussions. We also thank Roland Bauerschmidt for some useful comments. \corFK{We thank Ahmed El Alaoui, Heng Guo, Vishesh Jain, and the anonymous reviewers for useful feedback.}

\section{Poincar\'e Inequality}
\label{sec-poincare}
\corON{Recall that $\mu$ denotes the uniform measure on $\{\pm 1\}^n$, and that
for a matrix $J$ and a vector $h$, the Ising measure is defined as}
\begin{equation}\label{eqn:nu0}
d \nu_0(x) = \frac{1}{Z} e^{\frac{1}{2}\langle x, J x \rangle + \langle h, x \rangle} d \mu.
\end{equation}
We can clearly assume without loss of generality that $J$ is \corON{symmetric and} positive definite, which we 
do henceforth.
For any measure $\nu$ on $\{\pm 1\}^n$, we define the \emph{Dirichlet form}
\begin{equation}\label{eqn:dirichlet-form}
\cE_{\nu}(\varphi,\varphi) = \EE_{\nu} \sum_{i = 1}^n (\EE_{\nu}[\varphi(X) \mid X_{\sim i}] - \varphi(X))^2,
\end{equation}
where the associated generator of the Glauber dynamics is
\[ (\cL_{\nu} \varphi)(x) = \sum_{i = 1}^n \big(\EE_{\nu}[\varphi(X) \mid X_{\sim i} = x_{\sim i}] - \varphi(x)\big). \]
The main result of this section is a dimension-free Poincar\'e inequality for $\nu_0$ under the Glauber dynamics, provided that $\|J\|_{OP} < 1$. 
\begin{theorem}\label{thm:poincare}
For $\nu_0$ as in \eqref{eqn:nu0} with $0 \preceq J \prec \mathrm{Id}$ and any test function $\varphi : \{\pm 1\}^n \to \mathbb{R}$, we have the following inequality:
\[ (1 - \|J\|_{OP})\Var_{\nu_0}(\varphi(X)) \le \cE_{\nu_0}(\varphi,\varphi). \]
\end{theorem}
% \begin{theorem}
% There exist  universal constants $c,C>0$ such that if $\|J\|_{OP} \leq c$ then,
% for any test function $\varphi$,
% $$
% \mathrm{Var}_{\nu_0}[\varphi] \leq C n \mathrm{Lip}[\varphi]^2,
% $$
% \end{theorem}
% Here,
% \[\mathrm{Lip}[\varphi]=\max_{x\neq y\in \{\pm  1\}^n: d_H(x,y)=1} |\varphi(x)-\varphi(y)|\]
% is refered to as the Lipschitz norm of $\varphi$, and $d_H(x,y)$ denotes
% the Hamming distance between $x$ and $y$.
\begin{proof} The proof proceeds by a dynamical approach. It is
  clearly enough to consider $\varphi$ with Lipshitz norm $1$ and 
  $\varphi({\bf 1})=\varphi(( 1,\ldots,1))=0$, which implies that
$\|\varphi\|_\infty\leq n$.
 We will introduce a path of measures
\begin{equation}\label{eq:form}
d \nu_t(x) = e^{c_t + \frac{1}{2}\langle x, J_t x\rangle + \langle q_t + h, x \rangle } d \mu(x)=:F_t(x) d\nu_0(x),
\end{equation}
where $J_t,q_t$ are  processes, adapted
to the filtration  $\FF_t$ generated by an $n$-dimensional Brownian motion
$W_t$, and $c_t$ is a normalization
constant. For $\nu_t$ as in \eqref{eq:form}, introduce the barycenter
\begin{equation}
  \label{eq-satcor1}
a_t := \int x\ d \nu_t(x)
\end{equation}
and the test-function adjusted barycenter
\begin{equation}
  \label{eq-satcor2}
  V_t := \int \varphi(x) (x-a_t)\ d \nu_t\corON{(x)}.
\end{equation}

\RE{To define the process $F_t(x)$, we will first need the following technical result. Define by $\mathcal{H}$ the set of all linear subspaces of $\RR^n$.}

\begin{lemma}\label{lem:cprops}
	For every $\delta > 0$, there exists a function $C: \RR^n \times \mathcal{H} \to \mathbb{M}_{n \times n}$ which attains the following properties. For any linear subspace $H \subset \RR^n$,
	\begin{enumerate}
		\item For any $v$ the matrix $C(v,H)$ is positive semidefinite and $\mathrm{Im}(C(v,H)) \subseteq H$.
		\item The map $v \mapsto C(v,H)$ is Lipschitz continuous.
		\item If dim$(H)=d>1$
		then $\mbox{\rm Tr}(C(v,H))\geq d-1$.
		\item For any $v$ we have
		\begin{equation}
		\label{eq-bound1}
		 \corFK{|C(v,H) v| \le \delta}
		\end{equation}
		%\begin{equation}
		%\label{eq-bound1}
		%|C(v,H) v|= \phi(|v_1|)|\langle \hat v_1,v\rangle|
		%\leq |v_1|\phi(|v_1|)\leq \delta.
		%\end{equation}
	\end{enumerate}
\end{lemma}

\RE{If the continuity assumption is ignored, then one may simple take $C(v,H)$ as the \corFK{orthogonal} projection onto $H \cap v^\perp$, and for the sake on intuition, the reader may think of $C$ this way. Otherwise, the actual construction (and proof of the lemma) is postponed to subsection \ref{sec:technical}}.

Fix $\delta\ll 1$ (possibly depending on $n$), \RE{and let $C$ be the function provided by the above lemma}. For a \corFK{symmetric matrix $J$,}
%positive definite matrix $J$, 
introduce the subspace \corFK{$H_J$ spanned by the eigenvectors of $J$ corresponding to positive eigenvalues (when $J$ is positive semidefinite, this is $\mathrm{Im}(J)$).}
%% bad way to define H_J because not a subspace or closed
%\begin{equation}
%H_J = \{0\} \cup \{ v : \langle v, J v\rangle > 0 \} %\mathrm{Im}(J)
%\end{equation}
%provided $J$ is positive semidefinite.}

We are finally ready to introduce the dynamics for $F_t$, as the solution to the system of equations
\begin{eqnarray}
  \label{eq-satcor3}
  %&&\frac{dJ_t}{dt}= -2 \cdot C(V_t,H_{J_t})^2, \qquad J_0=J\\
  &&\frac{dJ_t}{dt}= -\ C(V_t,H_{J_t})^2, \qquad J_0=J\\
  \label{eq-satcor4}
  && 
%  Let $W_t$ be a Brownian motion adapted to a filtration $\FF_t$. Let $C_t$ be a matrix-valued process adapted to $\FF_t$, and consider the equation system
%$$
  d F_t(x) = F_t(x) \langle C(V_t,H_{J_t}) (x-a_t), d W_t \rangle, \qquad F_0(x) = 1, \forall x \in \{\pm 1\}^n,
\end{eqnarray}
see \eqref{eq-satcor1} and \eqref{eq-satcor2} for the definitions of 
$a_t$, $V_t$.
Note that the system in \eqref{eq-satcor3}-\eqref{eq-satcor4} is a stochastic
differential equation of dimension $2^n+n(n+1)/2$.
\corON{The existence and uniqueness of solutions to this system follow from
the next lemma}, \RE{whose proof is also postponed to Subsection \ref{sec:technical} below.} \corFK{In the proof of existence, we also show that $J_t$ remains positive semidefinite, which essentially follows from the fact $\mathrm{Im}(C(V_t,H_{J_t})) \subseteq H_{J_t}$.}
%In the following Lemma we show a solution exists.
\begin{lemma}\label{lem:existence}
  For any positive semidefinite matrix $J$, the system of stochastic differential equations \eqref{eq-satcor3}-\eqref{eq-satcor4} \corON{admits a unique strong solution.} \corFK{Furthermore, the matrix $J_t$ is positive semidefinite for all times $t \ge 0$, almost surely.}
  %satisfy global existence and uniqueness.
\end{lemma}

\corON{We continue with the proof of Theorem \ref{thm:poincare}.} 
\corFK{Define now $H_t = H_{J_t}$ and $C_t=C(V_t,H_t)$.}
%Note} 
\corFK{We start by verifying that the measure $\nu_t$ with density $d\nu_t(x)= F_t(x) d\nu_0(x) $ is indeed a probability measure on $\{\pm 1\}^n$. The nonnegativity of $F_t(x)$ is easily verified (it follows from \eqref{eq-bohen1} below) so it remains to check that the total mass of $\nu_t$ is $1$.
 Note that, due to \eqref{eq-satcor1} and \eqref{eq-satcor4}, we have for the total mass
\begin{equation}\label{eqn:normalizing-constant}
z_t := \sum_{x\in \{\pm 1\}^n} F_t(x) \nu_0(x)
\end{equation}
that
\[ dz_t = \sum_{x\in \{\pm 1\}^n} F_t(x) \langle C_t (x - a_t), dW_t \rangle \nu_0(x) = \langle C_t (a_t - z_t a_t), dW_t \rangle\]
and because $z_0 = 1$, by uniqueness of the solution we have $z_t = 1$ for all time, and hence $\nu_t$ is a probability measure as claimed.}
%   measure on $\{\pm 1\}^n$
%\corON{Note} that, due to \eqref{eq-satcor1} and \eqref{eq-satcor4},
%   we have that
%   \[ d\sum_{x\in \{\pm 1\}^n} F_t(x) \nu_0(x)=0,\]
%   and thus $d\nu_t(x)= F_t(x) d\nu_0(x) $ is indeed a probability
%   measure on $\{\pm 1\}^n$.

%$$
%with $d \nu_t = F_t(x) d \nu$ and $a_t = \int x d \nu_t(x)$. We will choose the adapted process $C_t$ later on. 

%Define now $C_t=C(V_t,H_t)$ and 
%Given a Lipschitz test-function $\varphi$, define
Define now
$$
M_t = \int \varphi(x) F_t(x)\ d \nu.
$$
We have from \eqref{eq-satcor4} that
\begin{equation}\label{eq:dM_t}
d M_t = \left \langle \int \varphi(x) C_t (x-a_t)\ d \nu_t, d W_t \right \rangle = \left \langle C_t V_t, d W_t \right \rangle,
\end{equation}
where $V_t := \int \varphi(x) (x-a_t) d \nu_t$. 
Note that 
\begin{equation}
  \label{eq-bound2}
  |d[M]_t/dt|=|C_t V_t|^2\leq \delta^2, \textrm{almost surely},
\end{equation}
see \eqref{eq-bound1}. 

Define the stopping time
$
T = \min \{t: \mathrm{rank}(J_t) \leq 1 \}
$
%Let $T$ be a stopping time defined later on. Under the assumption that the following holds true,
%\begin{equation}\label{eq:assump1}
%V_t \in \mathrm{Ker}(C_t), ~~ \forall t \in [0,T]
%\end{equation}
%we have that $[M]_T = 0$, almost surely. 
and let
$
Y_t := \mathrm{Var}_{\nu_t}[\varphi].
$
Then
$$
d Y_t = d \left (\int \varphi^2\ d\nu_t - M_t^2 \right ) = - d[M]_t + \mbox{martingale}.
$$
Consequently, 
we get from \eqref{eq-bound2} that
%assumption \eqref{eq:assump1} implies that
\begin{equation}\label{eq:consassump}
|\EE Y_T - Y_0| =|\EE Y_T- \mathrm{Var}_{\nu_0}[\varphi]|\leq \delta^2 \EE T.
\end{equation}
Next, Ito's formula gives
$
d \log F_t(x) = \langle C_t (x-a_t), d W_t \rangle - \corFK{(1/2)}|C_t (x-a_t)|^2 dt
$
so by integrating, we have
\begin{equation}
  \label{eq-bohen1}
\log F_t(x) = c_t + \langle q_t, x \rangle - \langle B_t x, x \rangle \corFK{/2}
\end{equation}
with $c_t, q_t$ being some Ito processes and with
\begin{equation}
  \label{eq-bohen2}
B_t = \int_0^t C_s^2\ ds
\end{equation}
(here we use that the matrix $C_t$ is symmetric). 
Note that, with $J_t$ as in \eqref{eq:form}, we obtain that
\begin{equation}
\label{eq-bohen3}
J_t = J - B_t,
\end{equation}
where $J$ is the original interaction matrix.

We next claim that
%: (i) The matrix $J_t$ is positive definite for all $t$ and (ii) 
almost surely, $T \leq \frac{1}{2} \mathrm{Tr}(J_0)=:T_0$. Indeed, 
%since $H_t \subset \mathrm{Im}(J_t)$, we have for all $\theta \in \mathrm{Ker}(J_t)$ that $d \langle \theta, J_t \theta \rangle = 0$. By continuity of the eigenvalues of $J_t$, we conclude (i). Now (ii) is a consequence of (i) and of the fact that 
$\frac{d}{dt} \mathrm{Tr}(J_t) \leq  - 2 \mathrm{dim}(H_t) + 2
\leq -2$ for all $t < T$, which means that if 
$T>T_0 $ then $\mathrm{Tr}(J_t)<0$, contradicting the positive
\corFK{semidefiniteness}%definiteness 
of $J_t$. We also deduce by monotonicity that
$$
0 \preceq J_T \preceq \|J\|_{OP} \mathrm{Id}.
$$
Thus, \eqref{eq:form} implies that
\begin{equation}\label{eq:formT}
  d \nu_T(x) = e^{c_T + \frac{1}{2} \langle U, x \rangle^2 + \langle q_T + h, x \rangle } d \mu(x)
\end{equation}
where $|U|^2 \leq \|J\|_{OP}$. 

To deduce the final result we use two more facts, proved below: Lemma~\ref{lem:poincare-rank-one}, which says the Poincar\'e inequality holds for the rank one model $\nu_T$, and Lemma~\ref{lem:supermartingale}, which says the Dirichlet form is a supermartingale under the dynamics \eqref{eq-satcor3}-\eqref{eq-satcor4}.
%It follows from the martingale property (more precisely, \eqref{eq:consassump}), we have
Given these facts, it follows from \eqref{eq:consassump} that
\[ (1-\|J\|_{OP})\Var_{\nu_0}(\varphi) \le (1 - \|J\|_{OP})\EE \Var_{\nu_T}(\varphi) + \delta^2 T_0 \le \EE \cE_{\nu_T}(\varphi,\varphi) + \delta^2 T_0 \le \cE_{\nu_0}(\varphi,\varphi) + \delta^2 T_0. \]
Taking $\delta \to 0$ proves the result. 
\end{proof}
\subsubsection{Proof of the existence of the process} \label{sec:technical}
\RE{In this section we prove the technical lemmas \ref{lem:cprops} and \ref{lem:existence}. }

\begin{proof}[Proof of Lemma \ref{lem:cprops}]
Introduce a smooth function $\phi: \RR_{+} \to [0,1]$ satisfying 
\begin{equation}
\label{eq-satcor0}
\phi(0)=1,\qquad \phi'(0) = 0, \qquad \sup_{z\in \RR_+}   z \phi(z)\leq \delta.
\end{equation}
For example, the function
\[\phi(z)= e^{- z^2/2\delta^2}\]
will do. Given a  
vector $v\in \RR^n$ and a linear subspace $H$ of $\RR^n$, write 
$v=v_1+v_2$ where $v_1\in H$ and $v_2\in H^\perp$, write $\hat v_1=v_1/|v_1|$,
and set
\begin{equation}\label{eq-cdef}
C(v,H)=\mathrm{Proj}_{H\cap v^\perp}+ \phi(|v_1|)\hat v_1\otimes \hat v_1.    
\end{equation} 
When $v_1 = 0$, $C(v,H)$ is just $\mathrm{Proj}_{H}$, \corFK{the orthogonal projection onto subspace $H$.}
The function $C(v,H)$ is a smooth approximation to the function 
$A(v,H)=\mathrm{Proj}_{H\cap v^\perp}$; the latter is not smooth owing to
a discontinuity when $|v_1|$ is small. Indeed, when $v=v_2 \in H^{\perp}$,
we note that $A(v,H)$  is the projection onto $H$, while if $v=\epsilon \hat v_1+ v_2$ for $\epsilon >0$, 
then the operator $A(v,H)$ is a projection onto a codimension 1 subspace of $H$.
On the other hand, $C(v,H)$ smoothes this transition, at the cost that
it is not a projection.	

\corFK{From the definition of $C(v,H)$ we have
\begin{equation*}
	|C(v,H) v|= \phi(|v_1|)|\langle \hat v_1,v\rangle| \leq |v_1|\phi(|v_1|)\leq \delta.
\end{equation*}
which shows the last claim \eqref{eq-bound1} in the lemma.}
	
We now justify the second claim of the lemma; the \corFK{remaining claims}follow directly from the definition. 
\corON{Rewrite}
%The second claim follows by rewriting
$C(v,H) = \mathrm{Proj}_H + (\phi(|v_1|) - 1) \hat{v}_1 \otimes \hat{v}_1$ and 
\corON{observe that} the first term is constant and the second term is Lipschitz in $v$: this is clear away from zero, and in a neighborhood of zero it follows by rewriting the second term as $\frac{\phi(|v_1|) - 1}{|v_1|^2} v_1 \otimes v_1$ and using that $\phi(0) = 1, \phi'(0) = 1$, and $\phi$ is smooth. 
\end{proof}

\begin{proof}[Proof of Lemma \ref{lem:existence}]
	Informally, both existence and uniqueness follow from the fact that $H_{J_t}$ will be piecewise constant. 
	%We now describe how to construct a solution more explicitly. 
	First we note that for a fixed subspace $H$, the equations
	\begin{eqnarray}
	\label{eq-fixedh1}
	&&\frac{dJ_t}{dt}= - C(V_t,H)^2 \\
	&& 
	\label{eq-fixedh2}
	d F_t(x) = F_t(x) \langle C(V_t,H) (x-a_t), d W_t \rangle \qquad \forall x \in \{\pm 1\}^n,
	\end{eqnarray}
	have Lipschitz coefficients (recall that a product of bounded Lipschitz functions is Lipschitz). \corON{Therefore, for any initial condition,
		a strong solution exists and is unique
		%and their solutions satisfy global existence and uniqueness for any initial conditions 
		\cite{karatzas1998brownian}.} Consider \eqref{eq-fixedh1}-\eqref{eq-fixedh2} with $H := H_{J} = \mathrm{Im}(J)$ 
		and initial conditions $J_0 = J, F_0(x) = 1$ and define \corON{the}
	stopping time $\tau_1 = \inf \{t \geq 0 : \mathrm{dim}(H_{J_t}) \le \mathrm{dim}(H_{J_0}) - 1 \}$. We use this system of equations to define $J_t, F_t$ on the interval $[0,\tau_1]$ and observe \corON{that}
	this solution satisfies \eqref{eq-satcor3}-\eqref{eq-satcor4}, 
	\corFK{provided we show that $H_{J_t} = H$
	%= \{0\} \cup \{v : \langle v, J_t v \rangle > 0\} = H$  %$H_{J_t} = \mathrm{Im}(J_t) = \mathrm{Im}(J_0)$ 
	for $t < \tau_1$, which we do next. First, observe that for $v \in \ker J$ that 
	\[ J_t v = \left(J + \int_0^t -C(V_s,H)^2 ds\right) v = 0 \]
	where the first equality is by \eqref{eq-fixedh1}, and the second equality uses that $J v = 0$ and $C(V_s,H) v = (v^T C(V_s,H))^T = 0$ using that  $C(V_s,H)$ is symmetric, $\mathrm{Im}(C(V_s,H)) \subseteq H$, and $v$ is in $\ker J$ which is the orthogonal complement of $H$. Thus, $\ker J \subseteq \ker J_t$ for all $t \le \tau_1$. Next, by the definition of $\tau_1$, we have for all $t < \tau_1$ that $\dim(H_{J_t}) = \dim(H_{J_0})$; by a dimension count, this implies that $\ker J = \ker J_t$, that $H_{J_t}$ has no negative eigenvalues, and finally that $H_{J_t} = H_{J_0}$ since they are both equal to the orthogonal complement of $\ker J$.}
	%To see that $H_{J_t} = H$ for $t < \tau_1$, first observe that $H_{J_t} \subset H_{J_0} = H$ because the matrix $J_t$ defined by \eqref{eq-fixedh1} decreases monotonically in the Loewner order. In the other direction, using that $\mathrm{Im}(C(V_t,H)) \subset H$ we see that the kernel of $J = J_0$ is contained in the kernel of $J_t$, and since $\dim(H_{J_t}) = \dim(H_{J_0})$ for $t < \tau_1$ this implies by a dimension count that all of the positive eigenvectors of $J_0$ are in $H_{J_t}$, hence $H = H_{J_0} \subset H_{J_t}$.} \corFK{Finally, note by continuity that since $J_t$ is positive semidefinite for $t < \tau_1$, it is also positive semidefinite at $t = \tau_1$.}
	
	 More generally, for all $i \le \mathrm{rank}(J_0)$ we define \corON{the stopping times}
	\[ \tau_i = \{t \ge \tau_{i - 1} : \mathrm{dim}(H_{J_t}) \le \mathrm{dim}(H_{J_{0}}) - i \} \]
	and define $J_t,F_t$ on $t \in [\tau_{i - 1}, \tau_i]$ by the solution to \eqref{eq-fixedh1}-\eqref{eq-fixedh2} with initial condition $J_{\tau_{i - 1}}, F_{\tau_{i - 1}}$ at time $\tau_{i - 1}$ \corFK{ and $H = H_{J_{\tau_{i - 1}}}.$} \corON{Finally,
		define the solution for  $t\geq \tau_{\mathrm{rank}(J_0)}$} similarly,
	with $H = \emptyset$ \corFK{(i.e. the solution is constant).} This shows existence \corFK{of the solution and positive semidefiniteness of $J_t$} and essentially the same argument proves uniqueness as well.
\end{proof}

\subsection{Rank one inequality}
In this section we prove the needed Poincar\'e inequality for rank one models (Lemma~\ref{lem:poincare-rank-one}).
We use the result of \cite{wu2006poincare}, which establishes a Poincar\'e inequality under a condition on the influence matrix referred to as the $\ell_2$-Dobrushin uniqueness regime (also studied in \cite{hayes2006simple,marton2019logarithmic}). 
%Whenever we write $x_{\sim i}$, this means the vector $x$ with coordinate $i$ removed. 
\begin{definition}
For two probability measures $\PP$ and $\QQ$ defined over the same measure space, their \emph{Total Variation (TV) Distance} is defined to be
\[ \|\PP - \QQ\|_{TV} := \sup_A |\PP(A) - \QQ(A)| \]
where $A$ ranges over all measurable events.
\end{definition}
\begin{definition}\label{defn:influence-matrix}
Suppose that $X$ is a random vector supported on a finite set $\mathcal{X}^n$ and distributed according to $\nu$. Define the \emph{influence matrix} $A$ to be the matrix with $\text{diag}(A) = 0$ and
\[ A_{ij} := \max_{x_{\sim i}, x'_{\sim i}} \|\PP_{\nu}[X_i \corFK{ = \cdot} | X_{\sim i} = x_{\sim i}] - \PP_{\nu}[X_i \corFK{= \cdot} | X_{\sim i} = x'_{\sim i}]\|_{TV} \]
where $x_{\sim i}$ and $x'_{\sim i}$ in $\mathcal{X}^{n - 1}$ are allowed to differ only in coordinate $j$, \corFK{and $\PP_{\nu}[X_i \corFK{ = \cdot} | X_{\sim i} = x_{\sim i}]$ denotes the conditional law of $X_i$ under $\nu$ given $X_{\sim i} = x_{\sim i}$.} We say that (the law of) $X$ satisfies the $\ell_2$-\emph{Dobrushin uniqueness condition} if $\|A\|_{OP} < 1$.
\end{definition}
\noindent
Note that in contrast to the interaction matrix $J$, the influence matrix $A$ has nonnegative entries.
We specialize the following Theorem to the setting of spins valued in $\{\pm 1\}$, though it holds in more general settings.
\begin{theorem}[Theorem 2.1 of \cite{wu2006poincare}]\label{thm:dobrushin-poincare}
Suppose that $X \sim \nu$ is a random vector valued in the hypercube $\{\pm 1\}^n$ and let $A$ be the corresponding influence matrix (as in Definition~\ref{defn:influence-matrix}). For any test function $\varphi : \{\pm 1\}^n \to \mathbb{R}$,
\[ (1 - \|A\|_{OP}) \Var(\varphi) \le \cE_{\nu}(\varphi,\varphi) \]
where $\cE_{\nu}(\varphi,\varphi)$ is the Dirichlet form associated to the Glauber dynamics under $\nu$.
%defined above.
\end{theorem}
To use this Theorem, we need to upper bound the spectral norm of the influence matrix for rank one models, which we do in the following Lemma.
\begin{lemma}\label{lem:influence-computation}
Suppose that 
\begin{equation}\label{eqn:rank-one-model}
d\nu(x) = \exp\left(\frac{1}{2} \langle x, u \rangle^2 + \langle h, x \rangle - c \right) d\mu(x)
\end{equation}
where $\mu$ is the uniform measure on $\{\pm 1\}^n$. The influence matrix $A$ of $\nu$ (from Definition~\ref{defn:influence-matrix}) satisfies $\|A\|_{OP} \le |u|^2$. 
\end{lemma}
\begin{proof}
First observe that
\[ \EE[X_i \mid X_{\sim i}] = \tanh(u_i \langle X_{\sim i}, u \rangle + h_i). \]
Therefore, from the definition of $A_{ij}$ \RE{and since $\tanh(\cdot)$ is $1$-Lipschitz}, we have
\begin{equation}\label{eqn:influence-bound}
A_{ij} = \frac{1}{2} \max_{x_{\sim i}, x'_{\sim i}} |\EE_{\nu}[X_i | X_{\sim i} = x_{\sim i}] - \EE_{\nu}[X_i | X_{\sim i} = x'_{\sim i}]| \le |u_i u_j| 
\end{equation}
where $x_{\sim i}, x'_{\sim i}$ range over vectors in $\{\pm 1\}^n$ differing only in coordinate $j$. Define $v$ to be the element-wise absolute value of $u$, i.e. $v_i = |u_i|$.
Since $A$ is a matrix with nonnegative entries, it follows from the Perron-Frobenius Theorem and \eqref{eqn:influence-bound} that $\|A\|_{OP} \le \| v v^T \|_{OP} = |v|^2 = |u|^2$.
\end{proof}
\noindent
Combining Lemma~\ref{lem:influence-computation} and Theorem~\ref{thm:dobrushin-poincare} yields the desired Poincar\'e inequality for rank one models.
\begin{lemma}\label{lem:poincare-rank-one}
Suppose that $\nu,u$ are as in \eqref{eqn:rank-one-model}.
%\[ d\nu(x) = \exp\left(\frac{1}{2} \langle x, u \rangle^2 + \langle h, x \rangle - c \right) d\mu(x) \]
%where $\mu$ is the uniform measure on $\{\pm 1\}^n$. 
Then for any test function $\varphi : \{\pm 1\}^n \to \mathbb R$,
\[ \corON{(1 - |u|^2)} \Var(\varphi) \le \cE_{\nu}(\varphi,\varphi) \]
where $\cE_{\nu}$ is the Dirichlet form associated to the Glauber dynamics under $\nu$. 
\end{lemma}
\subsection{\corON{The} Dirichlet form is a supermartingale}
\begin{lemma}\label{lem:supermartingale}
Let $W_t$ be a Brownian motion adapted to a filtration $\FF_t$. Let $C_t$ be a matrix-valued process adapted to $\FF_t$.
Let $\nu_0$ be an arbitrary measure on $\{\pm 1\}^n$ and suppose that $F_t$ is a solution to the SDE
\[ d F_t(x) = F_t(x) \langle C_t (x-a_t), d W_t \rangle, \qquad F_0(x) = 1, \forall x \in \{\pm 1\}^n \]
where $d\nu_t(x) = F_t(x) d\nu_0(x)$ and $a_t = \int x d\nu_t(x)$.
Let $\varphi : \{\pm 1\}^n \to \mathbb{R}$ be an arbitrary test function. Then the Dirichlet form $\cE_{\nu_t}(\varphi,\varphi)$ is a supermartingale. 
\end{lemma}
\begin{proof}
We use that the Dirichlet form can be rewritten as
\begin{equation}\label{eqn:dirichlet-form-useful}
\cE_{\nu_t}(\varphi,\varphi) =  \sum_{x \sim y} \frac{ \nu_t(x) \nu_t(y) }{\nu_t(x) + \nu_t(y)} (\varphi(x) - \varphi(y))^2 
\end{equation}
where $x \sim y$ denotes the adjacency relation on the hypercube, i.e. $x$ and $y$ differ in exactly one coordinate. To see this, consider arbitrary $\nu$, 
let $X$ and $Y$ be two independent samples from $\nu$, and observe
\begin{align} 
\cE_{\nu}(\varphi,\varphi) 
&= \EE_{\nu} \sum_{i = 1}^n \Var(\varphi(X) \mid X_{\sim i}) \nonumber \\
&= \frac{1}{2} \EE_{\nu} \sum_{i = 1}^n \EE_{\nu}[(\varphi(Y) - \varphi(X))^2 \mid Y_{\sim i} = X_{\sim i}, X_{\sim i}] \nonumber\\
&= \frac{1}{2} \EE_{\nu} \sum_{i = 1}^n \frac{(\varphi(Y) - \varphi(X))^2 \cdot \mathbf{1}[Y_{\sim i} = X_{\sim i}]}{\PP_{\nu}(Y_{\sim i} = X_{\sim i} \mid X_{\sim i})} \nonumber\\
&= \sum_{x \sim y} \nu(x) \nu(y) (\varphi(x) - \varphi(y))^2 \sum_{i = 1}^n \frac{\mathbf{1}[y_{\sim i} = x_{\sim i}]}{\PP_{\nu}(Y_{\sim i} = x_{\sim i})}\nonumber \\
&= \sum_{x \sim y} \frac{\nu(x) \nu(y)}{\nu(x) + \nu(y)} (\varphi(x) - \varphi(y))^2
\label{eq-alternate}
\end{align}
where in the second equality we used the identity $\Var(X) = \frac{1}{2} \EE[(X - Y)^2]$ for $Y$ an independent copy of $X$.
%\noindent

Given this, it suffices to show that $\frac{ \nu_t(x) \nu_t(y) }{\nu_t(x) + \nu_t(y)}$ is a supermartingale for fixed $x \sim y$.
Let us calculate the Ito differential $d \frac{\nu_t(x) \nu_t(y)}{\nu_t(x) + \nu_y(y)}$. We have by Ito's Lemma,
$$
d \log \nu_t(x) = \langle d W_t, \tilde x \rangle - \frac{1}{2} |\tilde x|^2 dt.
$$
and
$$
d \log \nu_t(y) = \langle d W_t, \tilde y \rangle - \frac{1}{2} |\tilde y|^2 dt.
$$
where $\tilde x = C_t(x-a_t)$, $\tilde y = C_t(y-a_t)$. Moreover,
$$
d \log \left (\nu_t(x) + \nu_t(y) \right ) = \langle d W_t, \alpha \tilde x + \beta \tilde y \rangle - \frac{1}{2} | \alpha \tilde x + \beta \tilde y  |^2 dt
$$
where $\alpha = \frac{\nu_t(x)}{\nu_t(x) + \nu_t(y)}$, $\beta = \frac{\nu_t(y)}{\nu_t(x) + \nu_t(y)}$. Therefore,
\begin{align*}
&d \left (\frac{\nu_t(x) \nu_t(y)}{\nu_t(x) + \nu_y(y)} \right ) \\
&= \left (\frac{\nu_t(x) \nu_t(y)}{\nu_t(x) + \nu_y(y)} \right ) \frac{1}{2} \left (| \alpha \tilde x + \beta \tilde y  |^2 + | \beta \tilde x + \alpha \tilde y  |^2 - |\tilde x|^2 - |\tilde y|^2 \right ) dt + \mbox{ martingale}.
\end{align*}
%$$
%d \log \left (\frac{\nu_t(x) \nu_t(y)}{\nu_t(x) + \nu_y(y)} \right ) = \langle d W_t, \beta x + \alpha y \rangle - \frac{1}{2} |\tilde x|^2 dt - \frac{1}{2} |\tilde y|^2 dt + \frac{1}{2} | \alpha \tilde x + \beta \tilde y  |^2 dt.
%$$
So again by Ito's Lemma, since $d e^{S_t} = e^{S_t} d S_t + \frac{1}{2} e^{S_t} d[S]_t$, we have
$$
d \left (\frac{\nu_t(x) \nu_t(y)}{\nu_t(x) + \nu_y(y)} \right ) = \left (\frac{\nu_t(x) \nu_t(y)}{\nu_t(x) + \nu_y(y)} \right ) \frac{1}{2} \left (| \alpha \tilde x + \beta \tilde y  |^2 + | \beta \tilde x + \alpha \tilde y  |^2 - |\tilde x|^2 - |\tilde y|^2 \right ) dt + \mbox{ martingale}.
$$
By convexity of $|\cdot|^2$ and since $\alpha + \beta = 1$, the above expression is a supermartingale.
\end{proof}
\section{Consequences for mixing time}
\label{sec-4}
By standard arguments which we now recall, the Poincar\'e inequality implies mixing time estimates for the Glauber dynamics. For a Markov semigroup $P_t = e^{t\Lambda}$, reversible with respect to $\nu$, a Poincar\'e inequality
\[ \gamma \Var_{\nu}[\varphi] \le \cE_{\nu}(\varphi,\varphi) \]
is equivalent to a spectral gap estimate:
\[ \gamma \le \lambda_1 - \lambda_2 \]
where $\lambda_1,\lambda_2$ are the top two eigenvalues of the transition rate matrix $\Lambda$ (see e.g. \cite{van2014probability}). The quantity $1/\gamma$ is known as the \emph{relaxation time} of the Markov chain. As usual, we let $P_t(\cdot,\cdot)$ denote the transition kernel of the Markov chain. Linear algebraic arguments establish the following mixing time estimate:
\begin{theorem}[Theorem 20.6 of \cite{levin2017markov}]\label{thm:relaxation-to-mixing}
  \corON{Let $P_t$ be a reversible Markov semigroup  over the} finite state space $\Omega$, with stationary measure $\pi$ and spectral gap $\gamma$.
  \corON{Then,}
\[ \max_{x \in \Omega} \|P_t(x,\cdot) - \pi\|_{TV} \le \epsilon \]
as long as 
\[ t \ge \frac{1}{\gamma} \log \frac{1}{\epsilon \min_{x \in \Omega} \pi(x)}. \]
\end{theorem}
\noindent
Applied to our situation, we have $\gamma = 1 - \|J\|_{OP}$ by Theorem~\ref{thm:poincare} and 
\[ \min_{x \in \{\pm 1\}^n} \nu_0(x) \ge 2^{-n} e^{-2n \|J\|_{OP} - 2|h|_1} \]
from the definition and H\"older's inequality. As a result, we obtain the following mixing time estimate for Glauber dynamics:
\begin{theorem}
  \label{theo-11}
For $P_t$ the continuous time Glauber dynamics on $\nu_0$ defined in \eqref{eqn:nu0},
\[ \max_{x \in \{\pm 1\}^n} \|P_t(x,\cdot) - \nu_0\|_{TV} \le \epsilon \]
as long as
\[ t \ge \frac{1}{1 - \|J\|_{OP}} \left((1 + 2 \|J\|_{OP})n + 2 |h|_1 + \log \frac{1}{\epsilon} \right). \]
\end{theorem}
\noindent One unit of time for the continuous dynamics corresponds to 
a \corON{Poissonian (with parameter $n$) number of} steps of the discrete-time Glauber dynamics. Correspondingly, Theorem~\ref{thm:poincare} implies an $O\left(\frac{n^2 + \|h\|_1 n + n\log(1/\epsilon)}{1 - \|J\|_{OP}}\right)$ mixing time estimate for the discrete time Glauber dynamics, using Theorem 12.3 of \cite{levin2017markov} in place of Theorem~\ref{thm:relaxation-to-mixing} above.

\section{A needle decomposition theorem} \label{sec:needle}
In this section we formulate a structural theorem, which follows as a byproduct of our proof. As mentioned above, this theorem may be thought of as an analogue to the technique due to Kannan, Lov\'asz and Simonovitz \cite{kannan1995isoperimetric} used in convex geometry (this technique was later generalized to the context of Riemannian manifolds, see \cite{Klartag-Needle}). It roughly states that measures with arbitrary quadratic potentials can be decomposed into mixtures of measures whose potentials are quadratic of rank one, in a way that: (i) The integral of some test function $\varphi$ is preserved, and (ii) the operator norm of the quadratic potential does not increase.

For $v,u \in \RR^n$, consider the measure $w_{v,u}$ on $\{\pm 1\}^n$ defined as
$$
\frac{d w_{u,v}}{d\mu}(x) =\frac{1}{Z_{u,v}} \exp\left(\frac{1}{2}\langle u, x \rangle^2 + \langle v, x \rangle\right)
$$
where $Z_{u,v} = \int_{\{\pm 1\}^n} e^{\frac{1}{2}\langle u, x \rangle^2 + \langle v, x \rangle} d \mu$.

Following roughly the same lines as the proof of Theorem \ref{thm:poincare} gives rise to the following result.
\begin{theorem} \label{thm:needledecomp}
Let $\nu_0$ be a probability measure on $\{\pm 1\}^n$ of the form
\begin{equation}\label{eq:formising2}
\frac{d\nu_0}{d\mu}(x) =\frac1Z \exp\left( \frac{1}{2} \langle x, J x \rangle + \langle h, x \rangle\right),
\end{equation}
where $J$ is positive definite, and let $\varphi: \{\pm 1\}^n \to \RR$. There exists a probability measure $m$ on $\RR^n \times \RR^n$ such that $\nu_0$ admits the decomposition
$$
\nu_0(A) = \int w_{u,v}(A) d m(u,v), ~~ \forall A \subset \{\pm 1\}^n,
$$
\corFK{with the properties that:}
\begin{enumerate}
\item  $m$-almost surely, $(u,v)$ are such that
$$
\int \varphi dw_{u,v} = \int \varphi d \nu_0,
$$
\item such that for $m$-almost every $(u,v)$, we have $|u| \leq \|J\|_{OP}$,
\item \corFK{and such that for all $x \sim y$ on the hypercube, we have the following inequality of conductances:
\[ \int \frac{w_{u,v}(x) w_{u,v}(y)}{w_{u,v}(x) + w_{u,v}(y)} dm(u,v) \le \frac{\nu_0(x) \nu_0(y)}{\nu_0(x) + \nu_0(y)}. \]
}
\end{enumerate}
\end{theorem}
\begin{proof} (sketch)
The decomposition follows by considering the evolution defined in \eqref{eq-satcor3} and \eqref{eq-satcor4} and defining the measure $m$ according to the decomposition implied by equation \eqref{eq:formT} above. The last property is a consequence of the corresponding supermartingale property shown in the proof of Lemma~\ref{lem:supermartingale}.
\end{proof}

The theorem can be used to reduce the concentration of a test function $\varphi$ over an Ising model to concentration over the measures $w_{u,v}$, as demonstrated by the following corollary.

\begin{corollary}
Let $K>0$ and let $\varphi: \{\pm 1\}^n \to \RR$ be a function such that for all $u,v \in \RR^n$ with $|u| \leq K$ one has  $\Var_{w_{u,v}}[\varphi] \leq 1$. Then for every $\nu_0$ of the form \eqref{eq:formising2} with $\|J\|_{OP} \leq K$, one has $\Var_{\nu_0}[\varphi] \leq 1$.
\end{corollary}
\begin{proof}
Applying Theorem \ref{thm:needledecomp}, the law of total variance implies that
\begin{align*}
\Var_{\nu_0}[\varphi] & = \int_{\RR^n \times \RR^n}  \left ( \int \varphi d w_{u,v} - \EE_{\nu_0}[\varphi] \right )^2 d m(u,v) + \int_{\RR^n \times \RR^n} \Var_{w_{u,v}}[\varphi] d m(u,v) \\ 
& = \int_{\RR^n \times \RR^n} \Var_{w_{u,v}}[\varphi] d m(u,v) \leq \sup_{|u| \leq K, v \in \RR^n} \Var_{w_{u,v}}[\varphi].
\end{align*}
The result follows readily.
\end{proof}

\corFK{Similarly, and analogous to the needle decomposition for convex sets, it allows us to establish functional inequalities for Ising models by reducing to the case of the rank one measures $w_{u,v}$; not just the Poincar\'e inequality, but also related inequalities such as the Log-Sobolev inequality. For the (bounded rate) Glauber dynamics, there is no uniform Log-Sobolev inequality over the class of models we consider, as there is no such inequality for biased product measures \cite{diaconis1996logarithmic} which are a special case. In subsequent work the Modified Log-Sobolev Inequality has been shown using this reduction \cite{anari2021entropic}.}
\section{Some examples}
\label{sec-examples}
\paragraph{Sherrington-Kirkpatrick (SK) Model.} This is the Ising model with $J$ symmetric and $J_{ij} \sim \mathcal{N}(0, \beta^2/n)$, i.e. up to rescaling $J$ is drawn from the Gaussian Orthogonal Ensemble. Letting the diagonal of $J$ be $0$, the spectrum of a GOE is contained in $[-2 - \epsilon,2 + \epsilon]$ asymptotically almost surely for any $\epsilon > 0$ (see e.g. \cite{anderson2010introduction}). Therefore our result implies the Poincar\'e inequality and polynomial time mixing for all $\beta < 1/4$. 

\begin{remark}
The Poincar\'e inequality (Theorem~\ref{thm:poincare}) applied to linear functions gives
\[ \Var(\langle w, X \rangle) \le \frac{1}{1 - \|J\|_{OP}} \EE_{\nu} \sum_{i = 1}^n (\EE_{\nu}[\langle w, X \rangle | X_{\sim i}] - \langle w, X \rangle)^2 \le \frac{1}{1 - \|J\|_{OP}} |w|^2 \]
using \eqref{eq-alternate},
and estimate \eqref{eqn:poincare-bb} from \cite{bauerschmidt2019very} gives a similar bound with a different constant.
Thus, both results imply that in the SK model for any fixed $\beta < 1/4$, $\|\Sigma\|_{OP} = O(1)$ with high probability\corFK{, where $\Sigma = \EE_{\nu} XX^T$ is the covariance matrix}. This partially verifies Conjecture 11.5.1 of \cite{talagrand2011mean} that $\|\Sigma\|_{OP} = O(1)$ for any $\beta < 1$.
\end{remark}

\paragraph{Diluted SK Model ($d$-Regular).} A variety of spin glass models on sparse graphs have been studied in the literature; one well-known ``diluted'' version of the SK model has the interaction matrix $J$ supported on a sparse Erd\H os-Reyni random graph --- see \cite{talagrand2010mean}. Along similar lines, we can consider a dilution where $J$ is supported on a random $d$-regular graph with $d \ge 3$. If we take a Rademacher disorder, i.e. $J_{ij} \sim Uni \{\pm \beta \}$ for $i,j$ neighbors and otherwise $J_{ij} = 0$, then it follows from a version of Friedman's Theorem that $\|J\|_{OP} \le \beta (2\sqrt{d - 1} + \epsilon)$ a.a.s. for any $\epsilon > 0$ --- see \cite{bordenave2015new,deshpande2019threshold,mohanty2020explicit}. It follows from our results that we have the Poincar\'e inequality and fast mixing for all $\beta < \frac{1}{4\sqrt{d - 1}}$, whereas the model is only in Dobrushin's uniqueness regime for $\beta = O\left(\frac{1}{d}\right)$ --- note that up to constants the latter bound is tight for general Ising models on arbitrary $d$-regular graphs \cite{galanis2012inapproximability,sly2012computational}.

\appendix
%\section{Dirichlet \corON{Forms} and other Discrete Gradients}\label{sec:dirichlet-comparison}
\section{Inequivalence of Dirichlet Forms}\label{apdx:dirichlet-comparison}
The Dirichlet form $\cE_{\nu}(\varphi,\varphi)$ can be viewed as the expected norm squared for an appropriate notion of gradient of $\varphi$. On the other hand, another natural notion of discrete gradient for functions on the hypercube is given by $(\nabla \varphi)_i(x) = \varphi(x_{\sim i}, x_i = 1) - \varphi(x_{\sim i}, x_i = -1)$ which is used in the result of Bauerschmidt and Bodineau \cite{bauerschmidt2019very}, and the norm of this discrete gradient squared is the Dirichlet form of a different semigroup with a variable transition rate \cite{martinelli1999lectures}. \corFK{In the case of the SK model, the transition rate of the variable-rate chain is sometimes exponentially large in $\sqrt{n}$; in what follows, we give a simple example of a function $\varphi$ witnessing that the Dirichlet forms similarly can differ in size by an exponentially large factor in $\sqrt{n}$.}

To compare these two Dirichlet forms, we have the following estimates which follow immediately from \eqref{eqn:dirichlet-form-useful}:
\[ \left(\min_{x \sim y} \frac{\nu(y)}{\nu(x) + \nu(y)}\right) \EE_{\nu} \|\nabla \varphi\|^2 \lesssim \cE_{\nu}(\varphi,\varphi) \lesssim \EE_{\nu} \|\nabla \varphi\|^2. \]
In the context of the SK Model, the parenthesized term is \corON{of}
size $e^{-\Theta(\beta \sqrt{n})}$ and both estimates are tight up to constants. To see this for the lower bound, define $a \in \{\pm 1\}^n$ by $a_1 = -1$ and $a_j = \text{sgn}(J_{1j})$ otherwise; the significance of this choice is that in the SK model, it's exponentially unlikely to see $X_1 = a_1$ given $X_{\sim 1} = a_{\sim 1}$. Let $\lambda$ be an atomic measure supported on $a$, so
\[ \frac{d\lambda}{d\nu}(x) = \frac{\mathbf{1}[x = a]}{\nu(a)}. \]
If we define $\varphi = \sqrt{\frac{d\lambda}{d\nu}}$ then, \corON{see \eqref{eq-alternate},}
\[ \cE\left(\varphi,\varphi\right) = \sum_{x \sim y} \frac{\nu(x)\nu(y)}{\nu(x) + \nu(y)} (\varphi(x) - \varphi(y))^2 = \sum_{y : y \sim a} \frac{\nu(y)}{\nu(a) + \nu(y)} \le n.  \]
In comparison, for the discrete gradient \corON{ $\nabla \varphi$} we have
\[ \EE |\nabla \varphi(X)|^2 \ge \frac{e^{c\beta \sqrt{n}} \nu(a)}{\nu(a)} = e^{\Theta(\beta \sqrt{n})}   \]
where the lower bound follows by considering the $a'$ which equals $a$ but flipped on the first coordinate, and which (from the definition of the SK model) is $e^{\Theta(\beta \sqrt{n})}$ more likely under $\nu$.

\bibliographystyle{plain}
\bibliography{bib}

\begin{thebibliography}{10}

\bibitem{anari2021entropic}
Nima Anari, Vishesh Jain, Frederic Koehler, Huy~Tuan Pham, and Thuy-Duong
  Vuong.
\newblock Entropic independence in high-dimensional expanders: Modified
  log-sobolev inequalities for fractionally log-concave polynomials and the
  ising model.
\newblock {\em arXiv preprint arXiv:2106.04105}, 2021.

\bibitem{anderson2010introduction}
Greg~W Anderson, Alice Guionnet, and Ofer Zeitouni.
\newblock {\em An introduction to random matrices}, volume 118.
\newblock Cambridge university press, 2010.

\bibitem{bauerschmidt2019very}
Roland Bauerschmidt and Thierry Bodineau.
\newblock A very simple proof of the {LSI} for high temperature spin systems.
\newblock {\em Journal of Functional Analysis}, 276(8):2582--2588, 2019.

\bibitem{bordenave2015new}
Charles Bordenave.
\newblock A new proof of {Friedman's} second eigenvalue theorem and its
  extension to random lifts.
\newblock {\em arXiv preprint arXiv:1502.04482}, 2015.

\bibitem{deshpande2019threshold}
Yash Deshpande, Andrea Montanari, Ryan O'Donnell, Tselil Schramm, and
  Subhabrata Sen.
\newblock The threshold for {SDP}-refutation of random regular {NAE-3SAT}.
\newblock In {\em Proceedings of the Thirtieth Annual ACM-SIAM Symposium on
  Discrete Algorithms}, pages 2305--2321. SIAM, 2019.

\bibitem{diaconis1996logarithmic}
Persi Diaconis and Laurent Saloff-Coste.
\newblock Logarithmic {Sobolev} inequalities for finite markov chains.
\newblock {\em The Annals of Applied Probability}, 6(3):695--750, 1996.

\bibitem{dobrushin-uniqueness}
Roland~Lvovich Dobrushin.
\newblock The description of a random field by means of conditional
  probabilities and conditions of its regularity.
\newblock {\em Theor. Prob. Appl.}, 13:197--224, 1968.

\bibitem{eldan2013thin}
Ronen Eldan.
\newblock Thin shell implies spectral gap up to polylog via a stochastic
  localization scheme.
\newblock {\em Geometric and Functional Analysis}, 23(2):532--569, 2013.

\bibitem{galanis2012inapproximability}
Andreas Galanis, Daniel Stefankovic, and Eric Vigoda.
\newblock Inapproximability of the partition function for the antiferromagnetic
  {Ising} and hard-core models.
\newblock {\em Combinatorics, Probability and Computing}, 25:500--559, 2016.

\bibitem{gross1975logarithmic}
Leonard Gross.
\newblock Logarithmic {Sobolev} inequalities.
\newblock {\em American Journal of Mathematics}, 97(4):1061--1083, 1975.

\bibitem{hayes2006simple}
Thomas~P Hayes.
\newblock A simple condition implying rapid mixing of single-site dynamics on
  spin systems.
\newblock In {\em 2006 47th Annual IEEE Symposium on Foundations of Computer
  Science (FOCS'06)}, pages 39--46. IEEE, 2006.

\bibitem{kannan1995isoperimetric}
Ravi Kannan, L{\'a}szl{\'o} Lov{\'a}sz, and Mikl{\'o}s Simonovits.
\newblock Isoperimetric problems for convex bodies and a localization lemma.
\newblock {\em Discrete \& Computational Geometry}, 13(3-4):541--559, 1995.

\bibitem{karatzas1998brownian}
Ioannis Karatzas and Steven~E Shreve.
\newblock {\em Brownian Motion and Stochastic Calculus}.
\newblock Springer, 1998.

\bibitem{Klartag-Needle}
Bo'az Klartag.
\newblock Needle decompositions in {R}iemannian geometry.
\newblock {\em Mem. Amer. Math. Soc.}, 249(1180):v+77, 2017.

\bibitem{ledoux1999concentration}
Michel Ledoux.
\newblock Concentration of measure and logarithmic {Sobolev} inequalities.
\newblock In {\em Seminaire de probabilites XXXIII}, pages 120--216. Springer,
  1999.

\bibitem{levin2017markov}
David~A Levin and Yuval Peres.
\newblock {\em Markov chains and mixing times}, volume 107.
\newblock American Mathematical Soc., 2017.

\bibitem{martinelli1999lectures}
Fabio Martinelli.
\newblock Lectures on {G}lauber dynamics for discrete spin models.
\newblock In {\em Lectures on probability theory and statistics
  ({S}aint-{F}lour, 1997)}, volume 1717 of {\em Lecture Notes in Math.}, pages
  93--191. Springer, Berlin, 1999.

\bibitem{marton2019logarithmic}
Katalin Marton.
\newblock Logarithmic {Sobolev} inequalities in discrete product spaces.
\newblock {\em Combinatorics, Probability and Computing}, 28(6):919--935, 2019.

\bibitem{mohanty2020explicit}
Sidhanth Mohanty, Ryan O'Donnell, and Pedro Paredes.
\newblock Explicit near-{Ramanujan} graphs of every degree.
\newblock In {\em Proceedings of the 52nd Annual ACM SIGACT Symposium on Theory
  of Computing}, pages 510--523, 2020.

\bibitem{sly2012computational}
Allan Sly and Nike Sun.
\newblock The computational hardness of counting in two-spin models on
  d-regular graphs.
\newblock In {\em 2012 IEEE 53rd Annual Symposium on Foundations of Computer
  Science}, pages 361--369. IEEE, 2012.

\bibitem{talagrand2010mean}
Michel Talagrand.
\newblock {\em Mean field models for spin glasses: Volume I: Basic examples},
  volume~54.
\newblock Springer Science \& Business Media, 2010.

\bibitem{talagrand2011mean}
Michel Talagrand.
\newblock {\em Mean Field Models for Spin Glasses: Volume II: Advanced
  replica-symmetry and low temperature}.
\newblock Springer, 2011.

\bibitem{van2014probability}
Ramon van Handel.
\newblock Probability in high dimension.
\newblock Technical report, Princeton Univ., NJ.
  https://web.math.princeton.edu/{$\sim$}rvan/APC550.pdf, 2016.

\bibitem{wu2006poincare}
Liming Wu.
\newblock Poincar{\'e} and transportation inequalities for {Gibbs} measures
  under the dobrushin uniqueness condition.
\newblock {\em The Annals of Probability}, 34(5):1960--1989, 2006.

\end{thebibliography}
%\begin{thebibliography}{GGM}
%\end{thebibliography}

\end{document}